\documentclass[11pt,letterpaper]{amsart}
\usepackage[usenames,dvipsnames]{color}
\usepackage{amsthm,amsfonts,amssymb,amsmath,amsxtra}
\usepackage{mathrsfs}
\usepackage{mathtools}
\usepackage{enumitem}
\usepackage[pdfpagelabels,pdftex,hidelinks]{hyperref}
\usepackage{cleveref}
\usepackage[all]{xy}
\usepackage{comment}
\usepackage{tikz-cd}

\hypersetup{
  colorlinks=true,
  urlcolor=blue,
  linkcolor=blue,
  citecolor=blue,
  breaklinks=true,
  bookmarksopen=true,
  bookmarksnumbered=true,
  pdfpagemode=UseOutlines,
  plainpages=false
}

\def\brk{{\breve k}}

\def\g{\gamma}
\def\a{\alpha}
\def\der{{\rm der}}

\newcommand{\ov}{\overline}

\newtheorem{theorem}{Theorem}
\newtheorem{proposition}[theorem]{Proposition}
\newtheorem{lemma}[theorem]{Lemma}
\newtheorem{corollary}[theorem]{Corollary}
\theoremstyle{definition}
\newtheorem{definition}[theorem]{Definition}
\newtheorem{remark}[theorem]{Remark}

\numberwithin{equation}{section}
\numberwithin{theorem}{section}
\setlist[itemize]{leftmargin=*,itemsep=\smallskipamount}
\setlist[enumerate]{leftmargin=*,itemsep=\smallskipamount}

\begin{document}

\title[Characteristic Functions of Parahoric Character Sheaves]
{Characteristic Functions of Parahoric Character Sheaves}
\date{}
\begin{abstract}
We establish a Jordan decomposition formula for the characteristic
functions of the character sheaves on parahoric subgroups defined in \cite{INY25}.  For a
sufficiently large $q$, these functions are $(-1)^{\dim G}$
times the corresponding deep level Deligne--Lusztig characters, extending
\cite{Lu90} to positive depth.  We also prove the orthogonality for the
generalized deep level Green functions and the character sheaves functions. Moreover, we obtain an explicit expression of characteristic functions of  simple character sheaves. As an application,
we present a sheaf-theoretic expression to the multiplicity of the
$G^{F^2}$ character restricted to
$G^{F}$.
\end{abstract}

\author{Zhihang Yu}
\address{The Second Affiliated Hospital of Chongqing Medical Univesity, Chongqing 400010, China}
\address{The National Center for Applied Mathematics in Chongqing, Chongqing 401331, China}
\email{yuzhihang@amss.ac.cn}

\maketitle
\section{Introduction}
Character sheaves form an important tool in the representation theory of finite reductive groups of Lie type. They were introduced and studied by Lusztig in \cite{Lusztig_85I} and a series of subsequent works, with the aim of providing a geometric framework for understanding the almost characters of finite reductive groups. In \cite{Lu90}, Lusztig proved the Frobenius trace of these character sheaves agrees with the corresponding Deligne–Lusztig character for $q$ good enough. This provides an important geometric approach to the study of character tables of finite reductive groups (for example \cite{shoji95}).

Later, Lusztig \cite{Lusztig_17} proposed an analogue of character sheaves for reductive groups over finite local rings. More recently, Bezrukavnikov and Chan \cite{BC24} developed a general theory of character sheaves on parahoric subgroups. In the generic setting, they proved that the characteristic functions of these character sheaves realize the deep level Deligne--Lusztig characters. More recently, Ivanov, Nie and Yu \cite{INY25} gave an alternative construction of these character sheaves in the framework of Yu's types and proved the correspondence in regular setting. We refer to \cite{Lusztig_04,CI_MPDL, ChenS_17, ChanO_21, CS23, Chan_siDL, Chan24, IvanovNie_24,IvanovNie_25, CO25, Yu26,IN26} and references therein for this topic.

A further motivation comes from the search for a geometric theory of
representations of $p$-adic groups. Deep level Deligne--Lusztig varieties
provide geometric realizations of supercuspidal representations in the
settings studied in \cite{Nie_24,CO25}. This raises the question of how
these realizations might fit into a theory of character sheaves on
$p$-adic groups.
Our comparison theorem connects character sheaves on parahoric subgroups
to the geometric constructions of supercuspidal representations.
This may provide
a step toward understanding what a theory of character sheaves for
$p$-adic groups could be.

The goal of the present article is to generalize Lusztig's comparison
result \cite{Lu90} to character sheaves on
parahoric subgroups.  This also extends \cite[Theorem C]{BC24} and
\cite[Theorem 1.4]{INY25} from regular characters to arbitrary characters.
We also establish orthogonality for generalized deep level Green functions
and show that simple character sheaves provide an orthonormal
basis of uniform characters. Moreover, we obtained an explicit expression of  characteristic functions of  simple character sheaves. Finally, we apply these results to subfield
symmetric multiplicities \cite{LusztigF2,Shoji07}
from $G^{F^2}$ to $G^F$

\subsection{Notation}
Let $k$ be a non-archimedean local field with residue field
$\mathbb{F}_q$ of characteristic $p$, throughout this paper, we suppose that $p$ is not bad \cite[\S 1]{Yu26}. Let $\brk$ be the completion of a
maximal unramified extension of $k$, and let $\mathcal{O}_{\brk}$ be its
ring of integers. Let $F$ be the Frobenius automorphism of $\brk$ over
$k$. Let $\mathbb{G}$ be a $k$-rational reductive group that splits over
$\brk$, and let $\mathbb T\subseteq\mathbb B\subseteq\mathbb G$ be a
$k$-rational maximal torus contained in a Borel subgroup.

Let $\mathcal{G}=\mathcal{G}_{{\bf x},0}$ be the parahoric model of
$\mathbb G$ attached to a point ${\bf x}$ in the apartment of $\mathbb T$
in the Bruhat--Tits building of $\mathbb G$ over $k$. For
$r\in\widetilde{\mathbb{R}}_{\geq0}$, where
$\widetilde{\mathbb{R}}:=\mathbb{R}\sqcup\{s+:s\in\mathbb{R}\}$, let
$\mathcal{G}_{{\bf x},r}$ be the $r$th Moy--Prasad subgroup of
$\mathcal{G}_{{\bf x},0}$. For $0\leq s\leq r$ in
$\widetilde{\mathbb{R}}$, we regard
\[
G^s_r=\mathcal{G}_{\mathbf{x},s}(\mathcal{O}_{\brk})/
      \mathcal{G}_{\mathbf{x},r+}(\mathcal{O}_{\brk})
\]
as an $\mathbb{F}_q$-rational perfectly smooth affine group scheme. Let
$\mathbb{H}\subseteq\mathbb{G}$ be a closed $\brk$-rational subgroup.
We denote by $H_r^s\subseteq G_r^s$ the closed subgroup defined in
\cite[\S 2.6]{CI_MPDL}.

Throughout the paper, an algebraic group means a linear algebraic
group.  For an element \(y\), we write \(y_s\) and \(y_u\) for its
semisimple and unipotent Jordan components, respectively.
\subsection{Character sheaves}
Consider 
\begin{equation*}
  \widetilde G_{r,{\rm vreg}} := \{(y,xT_r) \in G_{r, {\rm vreg}} \times G_r/T_r : x^{-1} y x \in T_r\}.
\end{equation*} 
where $G_{r,{\rm vreg}}$ denotes the set of very regular elements in
$G_r$.  Set $T_{r,{\rm vreg}}:=T_r\cap G_{r,{\rm vreg}}$ and consider
the maps
\begin{equation*}
\xymatrix@R=10pt@C=15pt{
    & \widetilde G_{r,{\rm vreg}} \ar@<2pt>[dl]_{\eta_{\rm{vreg}}} \ar@<2pt>[dr]^{\pi_{{\rm vreg}}} \\
    T_{r, {\rm vreg}} & & G_{r,{\rm vreg}}}
\end{equation*}
given by
\begin{equation*}
  \eta_{{\rm vreg}}(y,xT_r) = x^{-1} y x, \qquad \pi_{{\rm vreg}}(y,xT_r) = y.
\end{equation*}
For any local system $\mathcal{L}$ on $T_r$, these maps define a perverse
sheaf $\mathfrak{Ind}^{G_r}_{T_r}\mathcal{L}$ on $G_r$:
\[
\mathfrak{Ind}^{G_r}_{T_r}\mathcal{L}
=(j_{\mathrm{vreg}})_{!*}
 \bigl((\pi_{\mathrm{vreg}})_!\eta_{\mathrm{vreg}}^*
 \mathcal{L}_{\mathrm{vreg}}[\dim G_r]\bigr),
\]
where $j_{{\rm vreg}}:G_{r,{\rm vreg}}\hookrightarrow G_r$ is the
inclusion and $\mathcal{L}_{\mathrm{vreg}}$ is the restriction of
$\mathcal{L}$ to $T_{r,\mathrm{vreg}}$.
\subsection{Generalized deep level Green functions and Deligne--Lusztig characters}
Let $\ell \neq p$ be a prime. Let $\phi: T^F_r \to \ov {\mathbb{Q}}_\ell^\times$ be a character of depth $r\geq 0$ and let $\mathcal{L}_\phi$ denote the local system on $T_r$ attached to $\phi$.

We define the generalized deep level Green function
\[
 \mathcal{Q}_{T,G}^{\mathcal L+}:
 \{\text{unipotent elements of }G^F\}
 \longrightarrow\overline{\mathbb Q}_\ell
\]
by
\[
 \mathcal{Q}_{T,G}^{\mathcal L+}(u)
 :=\chi_{\mathfrak{Ind}^{G}_{T}\mathcal L}(u).
\]

\begin{theorem}[Jordan decomposition for characteristic functions]
Let $s\in G_r^F$ be semisimple, $H=Z_{G_r}^0(s)$, and let
$u\in H^F$ be unipotent.  For
$x\in G_r^F$ satisfying $x^{-1}sx\in T_r$, set
$T_x=xT_rx^{-1}\subset H$ and let $\mathcal L_{\phi,x}$ be the pullback
of $\mathcal L_\phi$ under $T_x\to T_r$, $t\mapsto x^{-1}tx$.  Then
\[
 \chi_{\mathfrak{Ind}^{G_r}_{T_r}\mathcal L_\phi}(su)
 =\frac{1}{|H^F|}
  \sum_{\substack{x\in G_r^F\\x^{-1}sx\in T_r}}
  \mathcal{Q}_{T_x,H}^{\mathcal L_{\phi,x}+}(u)\,
  \chi_{\mathcal L_\phi}(x^{-1}sx).
\]
\end{theorem}

Recall from \cite{CI_MPDL,Lusztig_04} that, for the Moy--Prasad quotients $G_r$
with $r\geq0$, there are analogues of the classical Deligne--Lusztig
varieties, called deep level Deligne--Lusztig varieties.  The variety
$X_r=X_{T,U,r}$ is essentially the inverse image of $U_r$ under the Lang
map $G_r\to G_r$.  Attached to $\phi$ is the virtual $G_r^F$-module
\[R_{T_r}^{G_r}(\phi) := \sum_i (-1)^i H_c^i(X_r, \ov {\mathbb{Q}}_\ell)[\phi]
\]
which we may also regard as a (virtual) character of $G_r^F$.

In \cite{Nie_24}, the author defined another virtual $G_r^F$-module attached to $\phi$, by a similar Deligne-Lusztig type variety $Z_r$:
\[\mathcal{R}_{T_r}^{G_r}(\phi) := \sum_i (-1)^i H_c^i(Z_r, \ov {\mathbb{Q}}_\ell)[\phi]
\]
This construction comes from a fixed "Howe factorization"\cite{Kaletha_19} for $\phi$. We have
\begin{theorem}[\Cref{LC}]
Assume that $q$ is sufficiently large. Then
\[\chi_{\mathfrak{Ind}^{G_r}_{T_r}\mathcal L_\phi}
 =(-1)^{\dim G_r}\mathcal R^{G_r}_{T_r}(\phi).
\]
\end{theorem}

If $r=0$ and ${\bf x}$ is hyperspecial, and $\phi$ is of depth $0$. Then $G_{0,{\rm vreg}}$ is exactly the regular semisimple locus in $G_0$ (as noted in \cite[Proposition 5.4]{CI_MPDL}), hence the theorem above can be seen as a generalization of comparison theorem in \cite{Lu90}.

\begin{theorem}\cite[Theorem]{Nie_24}
    When $\mathbb{T}$ is elliptic, then 
    \[\mathcal{R}_{T_r}^{G_r}(\phi)=R_{T_r}^{G_r}(\phi)\]
\end{theorem}

\begin{corollary}
    Assume that $q$ is sufficiently large. Then 
\[\chi_{\mathfrak{Ind}^{G_r}_{T_r}\mathcal L_\phi}
 =(-1)^{\dim G_r} R^{G_r}_{T_r}(\phi).
\]
\end{corollary}

\subsection{Orthogonality and Fourier
expansions }
In section 6, we first obtain the scalar products of
induced functions and generalized deep level Green functions, and then pass
to simple summands. Let
\[
 \mathcal K=\mathfrak{Ind}_{T_r}^{G_r}\mathcal L_\phi,
 \qquad
 W_{\mathcal L_\phi}
 =\operatorname{Stab}_{N_{G_r}(T_r)/T_r}(\mathcal L_\phi).
\]
By \cite[\S5]{INY25}, \(\mathcal K\) is semisimple perverse, and its
simple summands are obtained from the corresponding depth-zero summands.
If \(A\) is an
\(F\)-equivariant simple summand of \(\mathcal K\). By deep level comparison theorem, we have \Cref{fourier-prop}
\[
 \chi_A
 =\frac{(-1)^{\dim G_r}}{|W_{\mathcal L_\phi}|}
   \sum_{w\in W_{\mathcal L_\phi}}
   \operatorname{Tr}\bigl((\theta_w\sigma_A)^{-1},V_A\bigr)
   \mathcal R_{T_{r,w}}^{G_r}(\phi_w),
\]
where \(V_A=\operatorname{Hom}(A,\mathcal K)\). 

Let $\mathcal{C}^{\mathrm{UN}}(G_r^F)$ be the span of the deep level Deligne--Lusztig
characters named the space of uniform characters. For sufficiently large $q$,
the functions $\chi_A$, with $A$ ranging over the $F$-stable simple
induced sheaves up to isomorphism, form a basis of $\mathcal{C}^{\mathrm{UN}}$.
With normalized dual Weil structures, \Cref{simple-orth} gives
\[
 \bigl\langle\chi_A,\chi_{\mathbb D_{G_r}^0B}\bigr\rangle_{G_r^F}
 =\delta_{A,B}.
\]

These characteristic functions are also related to the well-understood depth zero almost characters. Let $A_0$ correspond to $A$.
We have
\[
 \chi_A=(-1)^{\dim G_r-\dim G^0_0}
 \operatorname{Ind}_{\mathcal K_{\phi,r}^F}^{G_r^F}
 \bigl(\kappa_\phi\otimes\operatorname{Inf}_{(G^0_0)^F}^{\mathcal K_{\phi,r}^F}\chi_{A_0}\bigr),
\]
where $\kappa_\phi$ is defined in \cite[\S7.1]{Nie_24}, which is closed related to the Weil--Heisenberg representation in Yu's construction \cite{Yu_01}.

\subsection{Subfield symmetric multiplicity}
The distinction problem for a pair $(\mathbb{G},\mathbb{G}^\theta)$ with an involution $\theta$
concerns the space $\operatorname{Hom}_{\mathbb{G}^\theta(k)}(\pi,1)$.
For supercuspidal representations,
Hakim and Murnaghan \cite{HM} reduce this problem to distinction of depth-zero data
with quadratic-character twists. The study of deep level Deligne--Lusztig
representations motivates a geometric study
of symmetric multiplicities at positive depth, see for example \cite{LN26}.

Here we consider the Galois case of this general problem.
More precisely, we apply the comparison and orthogonality results
to the subfield pair $(G_r^{F^2},G_r^F)$.
This pair is the finite-level counterpart of
$(\mathbb G(k_2),\mathbb G(k))$, where $k_2/k$ is the unramified
quadratic extension. For a representation $\rho$ of $G_r^{F^2}$,
the quantity of interest is
\[
 \dim\operatorname{Hom}_{G_r^F}(\rho,1)
 =\dim\rho^{G_r^F}
 =\frac{1}{|G_r^F|}\sum_{g\in G_r^F}\operatorname{Tr}(g,\rho).
\]

When $r=0$, this is the classical subfield problem studied by
Gow, Kawanaka, Prasad and Lusztig (see the introduction to
\cite{LusztigF2}). The same question incorporating
positive depth is closely related to the deep
level Deligne--Lusztig representations on symmetric spaces over finite
rings in \cite{LN26}. Our application supplies a sheaf-theoretic expression for
some multiplicities.

Let
$\phi$ be a character of \(T_r^{F^2}\) and let
\(\mathcal K=\mathfrak{Ind}_{T_r}^{G_r}\mathcal L_\phi\), equipped
with its \(F^2\)-Weil structure.  Set
\[
 m_{T_r,\phi}
 =\left\langle
   \operatorname{Res}_{G_r^F}^{G_r^{F^2}}
      \mathcal R_{T_r}^{G_r,F^2}(\phi),1
  \right\rangle_{G_r^F}.
\]
Under the comparison input above, \Cref{F2-simple,subfield-trace}
evaluate this expression as
\[
 m_{T_r,\phi}
 =\sum_{\substack{A\in\operatorname{Irr}(\mathcal K)^{F^2}\\
                    F^*A\cong\mathbb D_{G_r}^0A}}
           \operatorname{Tr}(\sigma_A,V_A),
 \qquad |m_{T_r,\phi}|\leq |W_{\mathcal L_\phi}|.
\]
In particular, when ${\phi}$ is a regular character, the corresponding irreducible component of deep level Deligne-Lusztig representation has invariant dimension zero if
\(F^*\mathcal K\not\cong\mathbb D\mathcal K\), and one if \(F^*\mathcal K\cong\mathbb D\mathcal K\).

\subsection{Outline}
In Section 2 we recall the construction of deep level character sheaves
used in \cite{INY25}.  In Section 3 we construct a neighborhood $\mathcal{U}$ of the
identity in \(H=Z_G^0(s)\) which separates the semisimple parts of \(s\)
and of nearby elements.  In Section 4 we decompose the incidence
varieties \(X_\mathcal{U}\) and \(Y\) into pieces indexed by
\(H\backslash\mathcal M/T\) and compare them with the corresponding
pieces of the \(s\)-twisted incidence varieties.  In Section 5 we
define generalized deep level Green functions and prove the Jordan
decomposition formula and the comparison theorem. Base on this, In Section 6 we prove the
orthogonality formulas for generalized deep level Green function and orthogonality for the characteristic function of simple induced sheaf. As an application, Section~7 we study the  subfield
symmetric multiplicity from \(G^{F^2}\) to \(G^F\).
\subsection*{Acknowledgments} We would like to thank Sian Nie for helpful discussions and encouragement. We are also grateful to Ben Liu and Pengcheng Li for fruitful discussions.  The arguments are inspired in part by Lusztig's method in \cite{Lus_Ch2}.

\subsection*{AI disclosure}
ChatGPT was used during the preparation of the manuscript as a assistant for checking arguments, improving the exposition and identifying mathematical and typographical errors. The mathematical ideas, conceptual framework, and proof strategies were developed by the authors. All AI suggestions were carefully verified by the authors, who take full responsibility for the manuscript. 

\section{Preliminaries}\label{pre}
Let $p\neq\ell$ be distinct primes, and let $\mathbb F_q$ be a finite
field of cardinality $q$ and characteristic $p$.

\subsection{Additional notation and conventions}\label{2.1}
Let $\phi:\mathbb{T}(k)\to\ov{\mathbb{Q}}_\ell^\times$ be a smooth character
of depth $r\geq0$. Then, by
\cite[Proposition 3.6.7]{Kaletha_19}, $\phi$ admits a Howe factorization
$(\mathbb{G}^i,\phi_i,r_i)_{-1\leq i\leq n}$, where the $\mathbb{G}^i$ form an ascending
sequence of $\breve k$-rational Levi subgroups of $\mathbb{G}$ with
$\mathbb{G}^{-1}=\mathbb{T}$ and $\mathbb{G}^n=\mathbb{G}$, where
$0=r_{-1}<r_0<\dots<r_{n-1}\leq r_n=r$, and where
$\phi_i:\mathbb{G}^i(k)\to\ov{\mathbb{Q}}_\ell^\times$ is a character of depth
$r_i$.
Write $s_i = r_i/2$ and
\[
\mathcal{K}_{\phi,r} = (G^0)_r (G^{1})^{s_0}_{r} \cdots (G^n)_{r}^{s_{n-1}} \quad \text{ and }\mathcal{K}^+_{\phi,r} = (G^0)^{0+}_r (G^{1})^{s_0+}_{r} \cdots (G^n)_{r}^{s_{n-1}+}.
\]
We fix a $\brk$-rational Borel subgroup $\mathbb B$ containing $\mathbb T$ such that each
$\mathbb G^i\subseteq \mathbb G$ is a standard Levi subgroup with respect to $\mathbb B$. Let
$\mathbb U\subseteq \mathbb B$ be the unipotent radical, and let $\ov{\mathbb{U}}$ be the opposite
unipotent subgroup. Consider the Iwahori-like subgroup
\[
 \mathcal{I}_{\phi,U,r}
 = (\mathcal{K}_{\phi,r}\cap U_r)T_r
   (\mathcal{K}_{\phi,r}^+\cap\ov U_r)
\]
constructed in \cite{Nie_24}. Set 
\begin{align*}
&T^i_\der = G^i_\der \cap T,\\
&T_{\phi, r} = (T^0_\der)_{0+:r} (T^1_\der)_{r_0+:r} \cdots (T^n_\der)_{r_{n-1}+:r},\\
&\bar T_r = T_r/T_{\phi,r}.
\end{align*}

Throughout the remainder of the paper, we fix a semisimple element
\(s\in G^F\) and use the notation
\[
 G=G_r,\qquad T=T_r,\qquad H=Z^0_G(s),\qquad W:=N_{G}(T)/T.
\]

We also fix a character $\phi:\mathbb{T}(k)\to\ov{\mathbb Q}_\ell^\times$ and set
\[
  I=\mathcal I_{\phi,U,r},\quad \tilde{T}=T_{\phi, r},
  \qquad \bar T=\bar T_r=T/\tilde{T},\qquad
 p:T\twoheadrightarrow\bar T.
\]
Let $\mathcal{L}=\mathcal{L}_\phi$ and $\bar{\mathcal{L}}$ denote the
local systems on $T$ and $\bar T$, respectively, attached to $\phi$.
Thus $p^*\bar{\mathcal{L}}\cong\mathcal{L}$.
\subsection{Very regular elements and induction functor}\label{tor}
\begin{definition}[{\cite{CI_MPDL}}]
  An element $\gamma\in G_{\mathbf{x},0}$ is called \textit{very regular} if
  \begin{enumerate}
    \item the identity component $\mathbb{T}_\g$ of the centralizer of $\g$ in $\mathbb{G}$ is a maximal torus,
    \item the apartment of  $\mathbb{T}_\g$ contains $\mathbf{x}$,
    \item $\a(\g)\not\equiv1\pmod{\varpi\mathcal{O}_\brk}$ for every root $\a$ of $\mathbb{T}_\g$ in $\mathbb{G}$.
  \end{enumerate}

An element of $G=G_r$ is called very regular if it is the image of a
very regular element of $G_{\mathbf{x},0}$.  The set $G_{{\rm vreg}}$ of
very regular elements is an $F$-stable open dense subset of $G$.
\end{definition}

Consider 
\begin{equation*}
  Y := \{(y,xT) \in G_{ \rm vreg} \times G/T : x^{-1} y x \in T\}.
\end{equation*} 
Set $T_{\mathrm{vreg}}:=T\cap G_{\rm vreg}$ and consider the maps
\begin{equation*}
\xymatrix@R=10pt@C=15pt{
    & Y \ar@<2pt>[dl]_{\eta} \ar@<2pt>[dr]^{\pi} \\
    T_{\rm vreg} & & G_{\rm vreg}}
\end{equation*}
given by
\begin{equation*}
  \eta(y,xT) = x^{-1} y x, \qquad \pi(y,xT) = y.
\end{equation*}
The map $\pi$ is a finite étale $W$-torsor.  This correspondence defines
the intersection cohomology complex
$\mathfrak{Ind}^{G}_{T}\mathcal{L}$ on $G$:
\[
\mathfrak{Ind}^{G}_{T}\mathcal{L}
=(j_{\mathrm{vreg}})_{!*}
 \bigl(\pi_!\eta^*\mathcal{L}_{\mathrm{vreg}}[\dim G]\bigr),
\]
where $j_{{\rm vreg}}:G_{\rm vreg}\hookrightarrow G$ is the inclusion
and $\mathcal{L}_{\mathrm{vreg}}$ is the restriction of $\mathcal{L}$ to
$T_{\mathrm{vreg}}$.  Two other constructions of the functor
$\mathfrak{Ind}^{G}_{T}$ appear in \cite{BC24,INY25}. We call the simple summand of the form $\mathfrak{Ind}^{G}_{T} \mathcal{L}$ as a simple induced sheaf.

\begin{remark}\label{supportH}
We may suppose $s\in T$. The same construction applies with $\mathbb{G}$ replaced by its reductive subgroup
$Z^0_{\mathbb{G}}(\tilde s)$ and $G_{\mathrm{vreg}}$ replaced by
$
H_{\mathrm{vreg}}:=G_{\mathrm{vreg}}\cap H
$, where $\tilde s$ is lifting of $s$ in $\mathbb{G}$ by \cite[Lemma 6.2]{CO_25} such that $(Z^0_{\mathbb{G}}(\tilde s))_r=H$.
The corresponding objects
$\eta$ and $\pi$ are replaced by $\eta_H$ and $\pi_H$, respectively,
and $\pi_H$ remains a finite \'{e}tale $W_H$-torsor. In particular, for
the corresponding induced complex on $H$, one has an isomorphism of
Weil perverse sheaves
$$
\mathfrak{Ind}^{H}_{T}\mathcal{L}
\cong
(j_{\mathrm{vreg}})_{!*}
\bigl(
\pi_{H!}\eta_H^*\mathcal{L}_{\mathrm{vreg}}[\dim H]
\bigr).
$$
\end{remark}
\subsection{Yu type construction}\label{iw-sec}
We recall the construction from \cite{INY25}. Consider
\[
 X=\{(y,xI)\in G\times G/I:x^{-1}yx\in I\}.
\]
 
We have the following commutative diagram
\begin{equation}\label{iw-diag}
\begin{tikzcd}
    T_{\mathrm{vreg}} \arrow[dd, "p"] &
    Y \arrow[l, "\eta"'] \arrow[r, "\pi"]\arrow[d, "\sim" sloped, "q"'] &
    G_{\mathrm{vreg}} \arrow[d, "\sim" sloped] \\
    &
    X_{\mathrm{vreg}} \arrow[r] \arrow[hookrightarrow, d] &
    G_{\mathrm{vreg}} \arrow[hookrightarrow, d] \\
    \bar{T} &
    X \arrow[r, "\psi"] \arrow[l, "\beta"'] &
    G
\end{tikzcd}
\end{equation}
where $p:T\twoheadrightarrow\bar T$ and
$p_I:I\twoheadrightarrow\bar T$ are the natural quotient maps, and
\[
q:Y\to X, \qquad (y,xT)\longmapsto(y,xI),
\]
is an isomorphism \cite[Lemma 5.3]{INY25}. The maps $\beta$ and $\psi$
are given by
\begin{equation*}
  \beta(y,xI) =p_{I} (x^{-1} y x), \qquad \psi(y,xI) = y.
\end{equation*}
Here $X_{\mathrm{vreg}}:=\psi^{-1}(G_{\mathrm{vreg}})$.
\begin{theorem}[{\cite[Theorem 5.4]{INY25}}]
\label{iw}
There is an isomorphism of perverse sheaves
\[
\mathfrak{Ind}^{G}_{T}\mathcal{L}
\cong \psi_!\beta^*\bar{\mathcal{L}}[\dim G].
\]
and the Weil structure of right hand side is induced by this isomorphism.
\end{theorem}

\subsection{Characteristic functions}
Let $X$ be an $\mathbb{F}_q$-variety with Frobenius endomorphism $F$,
and let $D(X)$ denote the bounded derived category of constructible
$\ell$-adic sheaves.  Suppose that $\mathcal F\in D(X)$ is equipped with
a Weil structure, that is, an isomorphism
$\varphi:F^*\mathcal F\xrightarrow{\sim}\mathcal F$.  Its
chracteristic function is
\[\chi_{\mathcal{F}, \varphi}: X(\mathbb{F}_q) \to \ov{\mathbb{Q}}_\ell, \quad x \mapsto \sum_i (-1)^i \mathrm{tr}(\varphi, \mathcal{H}^i(\mathcal{F})_x),\]
where $\mathcal{H}^i(\mathcal{F})_x$ is the stalk at $x$ of the $i$th
cohomology sheaf of $\mathcal{F}$.

\section{A Special Neighborhood}\label{sp}

\begin{lemma}\label{solv}
Let \(B\) be a connected solvable algebraic group over an
algebraically closed field, let \(R=R_u(B)\), and let \(S\) be a maximal
torus of \(B\). Then:

\begin{enumerate}
\item every semisimple element of \(B\) is conjugate by an element of \(R\) to
an element of \(S\);
\item if \(t,t'\in S\) and \(b^{-1}tb=t'\) for \(b\in B\), then \(t=t'\) and
\(b\in Z_B(t)\);
\item \(Z_B(t)\) is connected for every semisimple \(t\in B\).
\end{enumerate}
\end{lemma}

\begin{proof}

These are standard facts of the structure theory of connected
solvable algebraic groups; see, for example, \cite{Borel_91}.
\end{proof}

\begin{lemma}\label{U}
Let $s\in G^F$ be semisimple and write $H=Z^0_{G}(s)$. There is an open
subset $\mathcal U\subset H$ containing the identity element $e$ such that

\[
\begin{aligned}
&h\mathcal U h^{-1}=\mathcal U &&(h\in H),\\
&g\in\mathcal U\Longleftrightarrow g_s\in\mathcal U,\\
&F(\mathcal U)=\mathcal U,\\
&Z_{G}(s g_s)\subseteq Z_{G}(s) &&(g\in\mathcal U).
\end{aligned}
\]
\end{lemma}

\begin{proof}
Choose an $F$-defined faithful representation
\[
\iota:G\hookrightarrow \operatorname{GL}(V).
\]
Decompose $V$ into the eigenspaces of $\iota(s)$:

\[
V=\bigoplus_{\lambda\in\Lambda}V_\lambda,
\qquad \iota(s)|_{V_\lambda}=\lambda\,\mathrm{id}.
\]
$H$ preserves every $V_\lambda$. For $g\in H$, set
\(
P_{\lambda,g}(X)
 =\det\!\left(X-\lambda\,\iota(g)|_{V_\lambda}\right)
\)
and let
\[
\Delta(g)=\prod_{\{\lambda,\mu\}\subset\Lambda,\ \lambda\ne\mu}
 \operatorname{Res}_X(P_{\lambda,g},P_{\mu,g}).
\]

The coefficients of the polynomials $P_{\lambda,g}$ and $\Delta$
are regular functions on $H$. At the identity,
$P_{\lambda,e}=(X-\lambda)^{\dim V_\lambda}$; the factors belonging to distinct eigenvalues are coprime. Hence $\Delta(e)\ne0$, and
\[
\mathcal U:=\{g\in H:\Delta(g)\ne0\}
\]
is an open neighborhood of $e$.

If $h\in H$, conjugation by $\iota(h)$ preserves every block $V_\lambda$ and
does not change its characteristic polynomial; hence $\mathcal U$ is
$H$-conjugation invariant. Moreover, $\Delta(g)=\Delta(g_s)$. Since $s$ is $F$-fixed, $F$ permutes the eigenspace
blocks and sends their resultants to their Frobenius transforms. Hence $F(\mathcal U)=\mathcal U$.

It remains to prove the last statement. For $g\in\mathcal U$, the resultant
condition says that the spectra of
$\lambda\iota(g_s)|_{V_\lambda}$ and
$\mu\iota(g_s)|_{V_\mu}$ are disjoint whenever $\lambda\ne\mu$. Thus every
endomorphism commuting with the semisimple operator
\[
\iota(sg_s)=\iota(s)\iota(g_s)
\]
preserves each $V_\lambda$. It consequently commutes with $\iota(s)$, which
acts on $V_\lambda$ as the scalar $\lambda$. Intersecting this inclusion of
centralizers in $\operatorname{GL}(V)$ with the closed subgroup $\iota(G)$
gives
\[
Z_{G}(sg_s)\subseteq Z_{G}(s).
\]
\end{proof}

\begin{proposition}\label{Tsep}
For the neighborhood $\mathcal U$ defined in \Cref{U}, if
$g\in\mathcal U$, $x\in G$, and
\[
x^{-1}s g_s x\in T,
\]
then
\[
x^{-1}s x\in T,
\qquad
x^{-1}g_sx\in T.
\]
\end{proposition}

\begin{proof}
Set $c=x^{-1}s g_s x$. By \Cref{U}, we have
\[
T\subseteq Z_{G}(c)
=x^{-1}Z_{G}(sg_s)x
\subseteq x^{-1}Z_{G}(s)x
=Z_{G}(x^{-1}s x).
\]
It follows that $x^{-1}s x\in Z_{G}(T)=T$ and
\(
x^{-1}g_sx=(x^{-1}s x)^{-1}c\in T.
\)
\end{proof}

\begin{theorem}
\label{sep}
There is an \(H\)-conjugation-invariant, \(F\)-stable open neighborhood
\(\mathcal U\subset H\) of the identity element \(e\) such that
\begin{enumerate}
\item \(g\in\mathcal U\) if and only if \(g_s\in\mathcal U\). In particular, $\mathcal U$ contains every unipotent element of $H$;
\item if \(g\in\mathcal U\), \(x\in G\), and \(x^{-1}sgx\in T\), then
      \(x^{-1}sx,x^{-1}g_sx\in T\);
\item if \(g\in\mathcal U\), \(x\in G\), and \(x^{-1}sgx\in I\), then
      \(x^{-1}sx,x^{-1}g_sx\in I\).
\end{enumerate}
\end{theorem}
\begin{proof}
Take the neighborhood from \Cref{U}.  To prove
(2), apply \Cref{Tsep} to
$(x^{-1}sgx)_s=x^{-1}sg_sx$.  For (3), after conjugating by a suitable
element of $I$, we may assume that $x^{-1}sg_sx\in T$. Using
\Cref{Tsep} after this conjugation, then conjugating
back proves the theorem.
\end{proof}

\section{Geometric Structure}
\label{geom}
Recall that in \Cref{iw-sec}
\[
 X=\{(y,xI)\in G\times G/I:x^{-1}yx\in I\},
\]
and set
\[X_{\mathcal U}=\{(g,xI)\in X:g\in s\mathcal U\}.\]
Also set
\[
 \mathcal M=\{x\in G:x^{-1}sx\in T\},\qquad
 \widehat{\mathcal M}=\{x\in G:x^{-1}sx\in I\}.
\]
We write
\[
 \Gamma=H\backslash\mathcal M/T,
 \qquad
 \widehat\Gamma=H\backslash\widehat{\mathcal M}/I.
\]
Unless stated otherwise, we assume \(\mathcal M\ne\varnothing\).

\subsection{Analysis of the variety $X$}

\begin{theorem}[Richardson's orbit theorem]\label{Rich}
Let \(A\) be an algebraic group, let \(Q\subset A\) be a
closed subgroup, and let \(a\in A\) be semisimple. Then the fixed-point variety
\[
(A/Q)^a=\{xQ:x^{-1}ax\in Q\}
\]
is a finite disjoint union of $Z^0_{A}(a)$-orbits, and each orbit is
both open and closed.
\end{theorem}

\begin{proof}
This follows from \cite[Theorem A]{Richardson_1982}, applied with the
subgroup $S=\overline{\langle a\rangle}$.
\end{proof}

\begin{lemma}\label{lem4.2}
The natural map
\begin{equation*}
\Gamma=H\backslash\mathcal M/T
   \xrightarrow{\sim}
\widehat\Gamma=H\backslash\widehat{\mathcal M}/I        
\end{equation*}
is a bijection between finite sets. Moreover, \(\widehat\Gamma\) is
naturally the set of \(H\)-orbits on \((G/I)^s\).
\end{lemma}

\begin{proof}
If \(x\in\widehat{\mathcal M}\), choose \(i\in I\) such that
\(i^{-1}x^{-1}sxi\in T\). Thus \(xi\in\mathcal M\), and the map is
surjective.

For injectivity, suppose that \(x_1,x_2\in\mathcal M\) and
\(x_2=h x_1 i\) with \(h\in H\), \(i\in I\). Write
\[
t_j=x_j^{-1}sx_j\in T.
\]
 Then
\[
t_2=i^{-1}t_1i.
\]

It follows that \(t_1=t_2\). Thus \(i\in Z_I(t_1)\), and this centralizer
is connected. Therefore
\[
x_1Z_I(t_1)x_1^{-1}\subset Z^0_G(s)=H,
\]
so \(x_1ix_1^{-1}\in H\). Hence
\[
x_2=h(x_1ix_1^{-1})x_1
\]
lies in the same left \(H\)-coset as \(x_1\), proving injectivity.

By definition, $\widehat{\mathcal M}/I\cong(G/I)^s$, and finiteness
follows from \Cref{Rich}.
\end{proof}
For $\mathcal O\in\Gamma$, let $\widehat{\mathcal O}$ denote the
corresponding element of $\widehat\Gamma$, and let
$\mathscr O_{\widehat{\mathcal O}}$ be the corresponding $H$-orbit in
$(G/I)^s$.  Define
\[X_{\mathcal U,{\mathcal O}}
 =\{(sg,xI)\in X_{\mathcal U}:
       xI\in\mathscr O_{\widehat{\mathcal O}}\}.\]
\begin{proposition}\label{Xdec}
\begin{equation}\label{Xdec-eq}
X_{\mathcal U}
 =\coprod_{{\mathcal O}\in\Gamma}
   X_{\mathcal U,{\mathcal O}}             
\end{equation}
is a finite partition into subsets that are open and closed in
\(X_{\mathcal U}\). 
\end{proposition}

\begin{proof}
Let \((sg,xI)\in X_{\mathcal U}\). By definition,
\(x^{-1}sgx\in I\). Hence \(x^{-1}sx\in I\) by \Cref{sep}, so
\(x\in\widehat{\mathcal M}\). Consequently the sets in
\eqref{Xdec-eq} cover
\(X_{\mathcal U}\).

Let
\[
p_2:X_{\mathcal U}\longrightarrow (G/I)^s,
\qquad (sg,xI)\longmapsto xI.
\]
We have
\[X_{\mathcal U,\mathcal O}
   =p_2^{-1}(\mathscr O_{\widehat{\mathcal O}}). \]
By \Cref{Rich}, each
$\mathscr O_{\widehat{\mathcal O}}$ is open and closed in $(G/I)^s$.
Its inverse image under $p_2$ is therefore open and closed in
$X_{\mathcal U}$. This proves this proposition.
\end{proof}

For \(\mathcal O\in\Gamma\), choose representatives
\(x_{\mathcal O}\in\mathcal O\) compatibly with Frobenius which follows from Lang's Lemma, so that
\(F(x_{\mathcal O})=x_{F(\mathcal O)}\).  If
\(\mathcal O\in\Gamma^F\), 
\(x_{\mathcal O}\in\mathcal O^F\).  We write
\begin{align*}
&s_{\mathcal O}=x_{\mathcal O}^{-1}sx_{\mathcal O}\in T,
&T_{\mathcal O}=x_{\mathcal O}Tx_{\mathcal O}^{-1}\subset H,\\
&\bar{T}_{\mathcal O}=x_{\mathcal O}Tx_{\mathcal O}^{-1}/x_{\mathcal O}\tilde{T}x_{\mathcal O}^{-1},
&I_{\mathcal O}
 =H\cap x_{\mathcal O}Ix_{\mathcal O}^{-1}.        
\end{align*}
\begin{remark}
Note $I_{\mathcal O}$ can be seen as Yu type Iwahori subgroup as in \Cref{2.1}, by restricting the Howe data to $Z^0_{\mathbb{G}}(\tilde s)$. 
\end{remark}

Let $\mathcal L_{\mathcal O}$ be the pullback of $\mathcal L$ along
the map
\[
 c_{\mathcal O}:T_{\mathcal O}\longrightarrow T,
 \qquad t\longmapsto x_{\mathcal O}^{-1}tx_{\mathcal O},
\]
and define $\bar{\mathcal L}_{\mathcal O}$ similarly.  Let
\begin{align*}
&\pi_{\mathcal{O}}:Y_{\mathcal{O}}\to H_{\rm{vreg}},\quad &&\eta_{\mathcal{O}}:Y_{\mathcal{O}}\to T_{\mathcal{O},\rm{vreg}},\\
&\psi_{\mathcal{O}}:X'_{\mathcal{O}}\to H,  &&\beta_{\mathcal{O}}:X'_{\mathcal{O}}\to \bar{T}_{\mathcal{O}}.
\end{align*}
be the maps obtained from $(H,T_{\mathcal O},\mathcal L_{\mathcal O})$
in the same way that the maps in \Cref{iw-diag} are obtained from
$(G,T,\mathcal L)$, where $H_{\rm{vreg}}=H\cap G_{\rm{vreg}}$ (see also \Cref{supportH}). Then
\[
 \mathfrak{Ind}^{H}_{T_\mathcal{O}}\mathcal{L}_{\mathcal{O}}
 \cong
 ({\psi_{\mathcal{O}}})_!\beta^*_{\mathcal{O}}
 \bar{\mathcal{L}}_{\mathcal{O}}[\dim H].
\]

Let
\[
 \bar c_{\mathcal O}:\bar T_{\mathcal O}\longrightarrow\bar T,
 \qquad
 \bar t\longmapsto
 \overline{x_{\mathcal O}^{-1}t x_{\mathcal O}}
\]
be the map induced by conjugation, and write
\(\bar s_{\mathcal O}=p(s_{\mathcal O})\). Define
\(\bar{\mathcal L}'_{\mathcal O}\) to be the pullback of
\(\bar{\mathcal L}\) along
\[
 m_{\mathcal O}:\bar T_{\mathcal O}\longrightarrow\bar T,
 \qquad
 \bar t\longmapsto \bar s_{\mathcal O}\bar c_{\mathcal O}(\bar t).
\]
Since \(\bar{\mathcal L}\) is multiplicative, there is a canonical
isomorphism
\begin{equation}\label{Ltw}
 \bar{\mathcal L}'_{\mathcal O}
 \cong
 \bar{\mathcal L}_{\bar s_{\mathcal O}}
 \otimes \bar{\mathcal L}_{\mathcal O},
\end{equation}
where \(\bar{\mathcal L}_{\bar s_{\mathcal O}}\) denotes the constant
rank-one local system on \(\bar T_{\mathcal O}\) with fiber the stalk of
\(\bar{\mathcal L}\) at \(\bar s_{\mathcal O}\).
If \(\mathcal O\in\Gamma^F\), the local systems
$\bar{\mathcal{L}}'_\mathcal{O}$,
$\bar{\mathcal{L}}_{\mathcal{O}}$, and
$\bar{\mathcal{L}}_{\bar{s}_{\mathcal O}}$ admit their natural
Weil structures. Hence
\begin{equation*}
 \chi_{\bar{\mathcal L}'_{\mathcal O}}(t)
 =\chi_{\bar{\mathcal L}}(\bar{s}_{\mathcal O})
  \chi_{\bar{\mathcal L}_{\mathcal O}}(t)
 \qquad t\in \bar{T}_{\mathcal O}^F.
\end{equation*}
Let ${\mathcal{L}}'_{\mathcal{O}}$ be the pullback of
$\bar{\mathcal L}'_{\mathcal O}$ along the quotient map
${T}_{\mathcal{O}}\to \bar{T}_{\mathcal{O}}$. Similarly, we have
\begin{equation*}
 \chi_{{\mathcal L}'_{\mathcal O}}(t)
 =\chi_{{\mathcal L}}(s_{\mathcal O})
  \chi_{{\mathcal L}_{\mathcal O}}(t)
 \qquad t\in {T}_{\mathcal O}^F.
\end{equation*}
Consider also the twisted incidence variety
\[
Y'_{\mathcal O}
=\{(g,zT_{\mathcal O})\in s^{-1}H_{\rm{vreg}}\times H/T_{\mathcal{O}}:
   z^{-1}gz\in T_{\mathcal O}\}.
\]
Define maps
\begin{align}
\pi'_{\mathcal{O}}:Y'_{\mathcal{O}}&\longrightarrow
s^{-1}H_{\rm{vreg}},
& (g,zT_{\mathcal O})&\longmapsto g,\notag\\
\eta'_{\mathcal{O}}:Y'_{\mathcal{O}}&\longrightarrow
s^{-1}T_{\mathcal{O},\rm{vreg}},
& (g,zT_{\mathcal O})&\longmapsto z^{-1}gz.
\end{align}

Consider the open subset of $X'_{\mathcal O}$ given by
 \[
X'_{\mathcal U,\mathcal O}
 =\{(g,zI_{\mathcal O})\in
       \mathcal U\times H/I_{\mathcal O}:
       z^{-1}gz\in I_{\mathcal O}\}\subset X'_{\mathcal{O}}.  
\]
The group \(I_{\mathcal O}\) is the stabilizer of
\(x_{\mathcal O}I\) in \(H\), hence
\[
\theta_{\mathcal O}:H/I_{\mathcal O}\xrightarrow{\sim}
\mathscr O_{\widehat{\mathcal O}},
\qquad zI_{\mathcal O}\longmapsto zx_{\mathcal O}I.
\]
We obtain the following comparison.
\begin{proposition}\label{Xloc}
The map
\begin{equation}
\Phi_{\mathcal O}:X'_{\mathcal U,\mathcal O}
 \xrightarrow{\sim}X_{\mathcal U,{\mathcal O}},
\qquad
(g,zI_{\mathcal O})\longmapsto
(sg,zx_{\mathcal O}I).                               
\end{equation}
is an isomorphism.
\end{proposition}

\begin{proof}
The formula is independent of the representative of \(zI_{\mathcal O}\)
because \(I_{\mathcal O}\) is the stabilizer of
\(x_{\mathcal O}I\). Write
\[
a=x_{\mathcal O}^{-1}sx_{\mathcal O}\in T\subset I,
\qquad
b=x_{\mathcal O}^{-1}z^{-1}gzx_{\mathcal O}.
\]
Since \(z\in H\), one has
\[
x_{\mathcal O}^{-1}z^{-1}(sg)zx_{\mathcal O}=ab.    
\]

$ab$ belongs to \(I\) if and only if
\(b\in I\). On the other hand, \(z^{-1}gz\in H\), and hence
\begin{equation}
b\in I
\quad\Longleftrightarrow\quad
z^{-1}gz\in H\cap x_{\mathcal O}Ix_{\mathcal O}^{-1}
 =I_{\mathcal O}.\label{inc}
\end{equation}

Thus \Cref{inc} shows both that
\(\Phi_{\mathcal O}\) is well
defined and that the two incidence conditions are equivalent.

An explicit inverse of $\Phi_{\mathcal O}$ is given by
\[
X_{\mathcal U,{\mathcal O}}\longrightarrow
X'_{\mathcal U,\mathcal O},
\qquad
(y,xI)\longmapsto
\bigl(s^{-1}y,\theta_{\mathcal O}^{-1}(xI)\bigr).   
\]
\end{proof}

\subsection{Analysis of the variety $Y$}
Recall that \(G_{\mathrm{vreg}}\) is the very regular locus in $G$ and that
\[
  Y
 =\{(y,xT)\in G_{\mathrm{vreg}}\times G/T:x^{-1}yx\in T\},
 \qquad \pi(y,xT)=y.
\]
Similarly, set
\[
Y_{\mathcal U}=\pi^{-1}(H_{\mathrm{vreg}}\cap s\mathcal{U})
\]
and
\[
Y_{\mathcal U,\mathcal O}
=\{(y,xT)\in Y_{\mathcal U}:x\in\mathcal O\},
\]
and define
\[
Y'_{\mathcal U,\mathcal O}
=\{(g,zT_{\mathcal O})\in Y'_{\mathcal O}:g\in\mathcal U\}.
\]
   
Let \((sg,xT)\in Y_{\mathcal U}\). \Cref{U} gives
\(x^{-1}sx\in T\), hence \(x\in\mathcal M\). We therefore obtain a
partition
\[
Y_{\mathcal U}
 =\bigcup_{{\mathcal O}\in\Gamma}
   Y_{\mathcal U,{\mathcal O}}.
\]
These subvarieties are related as follows.
\begin{proposition}\label{Yloc}
For \(\mathcal O\in\Gamma\), there is an isomorphism
\begin{equation}\label{Yloc-eq}
\Lambda_{\mathcal O}:
Y'_{\mathcal U,\mathcal O}\longrightarrow
Y_{\mathcal U,\mathcal O},
\qquad
(g,zT_{\mathcal O})\longmapsto
(sg,zx_{\mathcal O}T).                              
\end{equation}
\end{proposition}
\begin{proof}
The displayed formula is well defined because $T_{\mathcal O}$ is the
stabilizer of $x_{\mathcal O}T$ in $H$. If
$(g,zT_{\mathcal O})\in Y'_{\mathcal U,\mathcal O}$, then
\[
 x_{\mathcal O}^{-1}z^{-1}(sg)zx_{\mathcal O}
 =s_{\mathcal O}\,x_{\mathcal O}^{-1}z^{-1}gzx_{\mathcal O}\in T,
\]
so the image lies in $Y_{\mathcal U,\mathcal O}$.

Conversely, let $(sg,xT)\in Y_{\mathcal U,\mathcal O}$ and write
$x=zx_{\mathcal O}t$ with $z\in H$ and $t\in T$. Then
$xT=zx_{\mathcal O}T$. By \Cref{sep},
$x^{-1}g_sx\in T$. Moreover, $(sg)_u=g_u$, and hence
\[
 x^{-1}g_ux=(x^{-1}sgx)_u\in T.
\]
Thus $x^{-1}gx\in T$, or equivalently
$z^{-1}gz\in T_{\mathcal O}$. This gives an inverse to
\eqref{Yloc-eq}.
\end{proof}

For every \(\mathcal O\in\Gamma\), set
\(
R_{\mathcal O}
 =\{t\in T_{\mathcal O}\cap\mathcal U:st\in G_{\mathrm{vreg}}\}.   
\)
Then \(R_{\mathcal O}\) is an open dense subset of
\(T_{\mathcal O}\cap\mathcal U\).
By definition, we have an isomorphism
\begin{equation*}
Y'_{\mathcal U,\mathcal O}
 \cong H\times^{T_{\mathcal O}}R_{\mathcal O}.
\end{equation*}
where \(T_{\mathcal O}\) acts on $R_{\mathcal O}$ by conjugation. Hence
$Y'_{\mathcal U,\mathcal O}$ is smooth and irreducible of dimension
$\dim H$. Similarly,
\[
X'_{\mathcal U,\mathcal O}
 \cong H\times^{I_{\mathcal O}}
       (I_{\mathcal O}\cap\mathcal U).
\]
is smooth, irreducible, and of dimension \(\dim H\).

These identifications fit into the following Cartesian diagram:
\begin{equation}
\begin{tikzcd}
    Y'_{\mathcal{U},\mathcal{O}} \arrow[r,"\sim" sloped, "\Lambda_{\mathcal{O}}"'] \arrow[hookrightarrow, d, "\iota_{\mathcal O}"'] &
    Y_{\mathcal{U},\mathcal{O}} \arrow[hookrightarrow, d] \\
    X'_{\mathcal{U},\mathcal{O}} \arrow[r, "\sim" sloped, "\Phi_{\mathcal{O}}"'] &
    X_{\mathcal{U},\mathcal{O}}
\end{tikzcd}
\end{equation}
The left vertical arrow is
\begin{equation*}
\iota_{\mathcal O}:
Y'_{\mathcal U,\mathcal O}\longrightarrow
X'_{\mathcal U,\mathcal O},
\qquad(g,zT_{\mathcal O})\longmapsto
(g,zI_{\mathcal O}).
\end{equation*}
Under $\Lambda_{\mathcal O}$ and $\Phi_{\mathcal O}$,
$Y'_{\mathcal U,\mathcal O}$ is identified both with
$Y_{\mathcal U,\mathcal O}$ and with a nonempty open dense subset of
$X'_{\mathcal U,\mathcal O}$. Moreover,
\[
 Y_{\mathcal U,\mathcal O}
 =Y_{\mathcal U}\cap X_{\mathcal U,\mathcal O}.
\]
Since the subsets $X_{\mathcal U,\mathcal O}$ are open and closed by
\Cref{Xdec}, the subsets
$Y_{\mathcal U,\mathcal O}$ are
also open and closed. Consequently,
\begin{equation}\label{Ydec}
Y_{\mathcal U}
 =\coprod_{{\mathcal O}\in\Gamma}
   Y_{\mathcal U,{\mathcal O}}
\end{equation}
is the decomposition of $Y_{\mathcal U}$ into irreducible components.

\begin{proposition}
For every $\mathcal O\in\Gamma$,
\[
 \pi(Y_{\mathcal U,\mathcal O})
 =H_{\mathrm{vreg}}\cap s\mathcal U
\]
and
\[
 \pi'_{\mathcal O}(Y'_{\mathcal U,\mathcal O})
 =s^{-1}H_{\mathrm{vreg}}\cap\mathcal U.
\]
\end{proposition}
\begin{proof}
The inclusions from left to right follow from the definitions. Fix
$g\in s^{-1}H_{\mathrm{vreg}}\cap\mathcal U$ and set $y=sg$, so $y$ is very
regular in $G$. There exists $a\in G$ such that $a^{-1}ya\in T$. By \Cref{sep},
$a^{-1}sa\in T$ as well. Consequently
$S:=aTa^{-1}$ contains both $s$ and $g=s^{-1}y$, hence $S\subseteq Z_G^0(s)=H$.

Note that $S$ and $T_{\mathcal O}$ are conjugate under $H$.
Let $S_0$ and $T_{\mathcal O,0}$ be their maximal tori.
These are maximal tori contained in $H$, hence are maximal
tori of the connected group $H$. There exists $z\in H$ with
$z^{-1}S_0z=T_{\mathcal O,0}$. 
Taking centralizer on both sides yields
$z^{-1}Sz=T_{\mathcal O}$. In particular, $z^{-1}gz\in T_{\mathcal O}$.
Thus $(g,zT_{\mathcal O})\in Y'_{\mathcal U,\mathcal O}$, proving
surjectivity of $\pi'_{\mathcal O}$ onto
$s^{-1}H_{\mathrm{vreg}}\cap\mathcal U$. Finally, $\pi\circ\Lambda_{\mathcal O}
=\iota_s\circ\pi'_{\mathcal O}$, so this also holds for $\pi$. 
\end{proof}

\section{Characteristic Functions of Character Sheaves}
\label{char}
\subsection{Generalized deep level Green functions}
Recall the definition of $\mathfrak{Ind}^{G}_{T}\mathcal{L}$ from
\Cref{tor}. We define the generalized deep level Green function
\[
 \mathcal{Q}_{T,G}^{\mathcal L+}:
 \{\text{unipotent elements of }G^F\}
 \longrightarrow\overline{\mathbb Q}_\ell
\]
by
\[
 \mathcal{Q}_{T,G}^{\mathcal L+}(u)
 :=\chi_{\mathfrak{Ind}^{G}_{T}\mathcal L}(u).
\]

Using the decomposition $T=T_0\times T^{0+}$, we correspondingly write
$\mathcal{L}=\mathcal{L}_0\boxtimes\mathcal{L}^{0+}$. We likewise write
$\bar T=T_0\times\bar T^{0+}$, where
$\bar T^{0+}=T^{0+}/\tilde T$.
\begin{proposition}\label{Qpos}
The function $\mathcal{Q}_{T,G}^{\mathcal{L}+}$ depends only on
$\mathcal{L}^{0+}$. In other words, it is independent of the choice of
$\mathcal{L}_0$.
\end{proposition}
\begin{proof}
The argument is similar to \cite[8.3.2]{Lus_Ch2}. We use the realization in \Cref{iw-sec}:
\[
 \mathfrak{Ind}^{G}_{T}\mathcal L
 \cong \psi_!\beta^*\bar{\mathcal L}[\dim G].
\]
Let $u\in G^F$ be unipotent. The stalk formula for compactly supported
pushforward identifies the stalk at $u$ with
\[
 R\Gamma_c\bigl(\psi^{-1}(u),
                  \beta^*\bar{\mathcal L}\bigr)[\dim G].
\]
For $(u,xI)\in\psi^{-1}(u)$, the element $x^{-1}ux\in I$ is
unipotent.  Since the quotient map $p_I:I\twoheadrightarrow\bar T$
preserves Jordan decomposition, $\beta(u,xI)=p_I(x^{-1}ux)$ is a
unipotent element of $\bar T$. Consequently
\[
 \beta\bigl(\psi^{-1}(u)\bigr)\subseteq \bar T^{0+}.
\]
The Weil complex
$\mathfrak{Ind}^{G}_{T}\mathcal L|_{G_{\mathrm{unip}}}$ is only determined by
$\mathcal L|_{T^{0+}}$.
\end{proof}
\begin{remark}\label{changehowe}
    Let $\phi'$ be another character such that $\phi_{|{T^{0+}}}=\phi'_{|{T^{0+}}}$ For a Howe factorization $(\mathbb{G}^i,\phi_i,r_i)_{-1\leq i\leq n}$ of $\phi$, we only need to replace $\phi_{-1}$ by $\phi'_{-1}$ which is another charatcer of $T$ of depth $0$ and take $\phi_{i}=\phi'_{i}\quad 0\leq i$. Then $(\mathbb{G}^i,\phi'_i,r_i)_{-1\leq i\leq n}$ is a Howe factorization of $\phi'$. Hence in the proof of \Cref{Qpos}, we don't need consider the change of Iwahori type group $I$.
\end{remark}
\subsection{Deep level Lusztig's comparison theorem}
By \Cref{Yloc,Ydec}, we have the
commutative diagram
\begin{equation}
\begin{tikzcd}
     \coprod_{\mathcal{O}\in \Gamma}Y'_{\mathcal{U},\mathcal{O}} \arrow[r,"\sim" sloped] \arrow[d,"\coprod_{\mathcal{O}\in \Gamma}{\pi'_{\mathcal{O}}}"] &
   Y_{\mathcal{U}}\arrow[d, "\pi"] \\
     s^{-1}H_{\rm{vreg}}\cap \mathcal{U} \arrow[r, "\sim" sloped, "\iota_s"'] &
    H_{\rm{vreg}}\cap s\mathcal{U}
\end{tikzcd}
\end{equation}
where the upper isomorphism is the disjoint union of the maps
$\Lambda_{\mathcal O}$, and
\[
\iota_s(g)=sg.
\]

The construction is Frobenius equivariant: Frobenius carries the diagram
indexed by $\mathcal O$ to the one indexed by $F(\mathcal O)$.  It follows
that we have a canonical isomorphism of local systems over
$s^{-1}H_{\rm{vreg}}\cap \mathcal{U}$:

\begin{equation}\label{loc0}
 (\iota_s)^*\left(
   \left.\pi_!\eta^*\mathcal L_{\rm vreg}
   \right|_{H_{\rm vreg}\cap s\mathcal U}
 \right)
 \cong
 \bigoplus_{\mathcal O\in\Gamma}
 \left.(\pi'_{\mathcal O})_!(\eta'_{\mathcal O})^*
 \mathcal L'_{\mathcal O}
 \right|_{s^{-1}H_{\rm vreg}\cap\mathcal U}.
\end{equation}
This isomorphism is compatible with the chosen Weil structures.

For $\mathcal O\in\Gamma$, define the twisted complex on $H$ by
\[
 \mathcal K'_{\mathcal O}
 :=({\psi_{\mathcal O}})_!
     \beta_{\mathcal O}^*\bar{\mathcal L}'_{\mathcal O}[\dim H].
\]

\begin{proposition}\label{tw}
Let $\mathcal C_{s_{\mathcal O}}$ be the constant local system
on $H$ whose fiber is the stalk $\mathcal L_{s_{\mathcal O}}$. There is
a canonical isomorphism
\begin{equation}\label{tw-eq}
 \mathcal K'_{\mathcal O}
 \cong
 \mathcal C_{s_{\mathcal O}}\otimes
 \mathfrak{Ind}^{H}_{T_{\mathcal O}}\mathcal L_{\mathcal O}.
\end{equation}
If $\mathcal O\in\Gamma^F$, this isomorphism is compatible with the
natural Weil structures. In particular, for every unipotent element
$u\in H^F$,
\begin{equation}\label{tw-tr}
 \chi_{\mathcal K'_{\mathcal O}}(u)
 =\chi_{\mathcal L}(s_{\mathcal O})
  \mathcal{Q}_{T_{\mathcal O},H}^{\mathcal L_{\mathcal O}+}(u).
\end{equation}
\end{proposition}
\begin{proof}
By \Cref{Ltw},
\[
 \bar{\mathcal L}'_{\mathcal O}
 \cong
 \bar{\mathcal L}_{\bar s_{\mathcal O}}
 \otimes\bar{\mathcal L}_{\mathcal O}.
\]
The first factor on the right is constant, and
\[
 \beta_{\mathcal O}^*\bar{\mathcal L}_{\bar s_{\mathcal O}}
 \cong\psi_{\mathcal O}^*\mathcal C_{s_{\mathcal O}}.
\]
The projection formula therefore gives
\begin{align*}
 \mathcal K'_{\mathcal O}
 &\cong
 \mathcal C_{s_{\mathcal O}}\otimes
 ({\psi_{\mathcal O}})_!
 \beta_{\mathcal O}^*\bar{\mathcal L}_{\mathcal O}[\dim H]\\
 &\cong
 \mathcal C_{s_{\mathcal O}}\otimes
 \mathfrak{Ind}^{H}_{T_{\mathcal O}}\mathcal L_{\mathcal O},
\end{align*}
where the second isomorphism is the comparison theorem recalled in
\Cref{iw-sec}, applied to $(H,T_{\mathcal O})$.

If $\mathcal O$ is $F$-stable, we have $x_{\mathcal O}\in\mathcal O^F$.
Then every map used above is defined over $\mathbb F_q$, and the
multiplicativity isomorphism for $\bar{\mathcal L}$ is compatible with
Frobenius. Thus \eqref{tw-eq} is an isomorphism of Weil
complexes. Taking Frobenius traces and using
$\chi_{\bar{\mathcal L}}(\bar s_{\mathcal O})
=\chi_{\mathcal L}(s_{\mathcal O})$ proves
\eqref{tw-tr}.
\end{proof}

By the definition of $\mathfrak{Ind}^{G}_{T}\mathcal{L}$ and
\Cref{tw}, we may rewrite
\eqref{loc0} as
\begin{equation}\label{loc}
 \left.(\iota_s)^*\mathfrak{Ind}^{G}_{T}\mathcal L
 \right|_{s^{-1}H_{\rm vreg}\cap\mathcal U}[-\dim G]
 \cong
 \bigoplus_{\mathcal O\in\Gamma}
 \left.
  \mathcal K'_{\mathcal O}
 \right|_{s^{-1}H_{\rm vreg}\cap\mathcal U}[-\dim H].
\end{equation}
Assume that we can show that the isomorphism \ref{loc} is the restriction of an isomorphism:
\begin{equation}\label{glob}
 (\iota_s)^*\left(
   \left.\mathfrak{Ind}^{G}_{T}\mathcal L\right|_{s\mathcal U}
 \right)[-\dim G]
 \cong
 \bigoplus_{\mathcal O\in\Gamma}
 \left.
 \bigl(\mathcal C_{s_{\mathcal O}}\otimes
       \mathfrak{Ind}^{H}_{T_{\mathcal O}}\mathcal L_{\mathcal O}\bigr)
 \right|_{\mathcal U}[-\dim H].
\end{equation}
The isomorphism \ref{glob} extending \ref{loc} is unique and is automatically compatible with the Weil structure. This follows from properties of intersection cohomology complexes.

Indeed, \Cref{iw-sec} and the
decomposition from \Cref{Xdec} give
\[
 \left.\psi_!\beta^*\bar{\mathcal L}\right|_{s\mathcal U}
 \cong
 \bigoplus_{\mathcal O\in\Gamma}
 (\psi|_{X_{\mathcal U,\mathcal O}})_!
 \left.\beta^*\bar{\mathcal L}\right|_{X_{\mathcal U,\mathcal O}}.
\]
Under the isomorphism $\Phi_{\mathcal O}$ of
\Cref{Xloc}, one has
\begin{align*}
 \beta\bigl(\Phi_{\mathcal O}(g,zI_{\mathcal O})\bigr)
 &=p_I\bigl(x_{\mathcal O}^{-1}z^{-1}(sg)zx_{\mathcal O}\bigr)\\
 &=\bar s_{\mathcal O}\,
   \bar c_{\mathcal O}\bigl(\beta_{\mathcal O}(g,zI_{\mathcal O})\bigr).
\end{align*}
Thus $\Phi_{\mathcal O}^*\beta^*\bar{\mathcal L}
=\beta_{\mathcal O}^*\bar{\mathcal L}'_{\mathcal O}$. Transporting
compactly supported pushforward along $\Phi_{\mathcal O}$ yields
\[
 (\iota_s)^*
 \left(\left.\psi_!\beta^*\bar{\mathcal L}\right|_{s\mathcal U}\right)
 \cong
 \bigoplus_{\mathcal O\in\Gamma}
 \left.({\psi_{\mathcal O}})_!
 \beta_{\mathcal O}^*\bar{\mathcal L}'_{\mathcal O}
 \right|_{\mathcal U}.
\]
After inserting the normalizing shifts and applying
\Cref{tw}, this is precisely
\eqref{glob}. All the constructions
are Frobenius equivariant, Frobenius sends the summand indexed by
$\mathcal O$ to the one indexed by $F(\mathcal O)$.

\begin{proposition}\label{Jsum}
For every unipotent element \(u\in H^F\),
\[\chi_{\mathfrak{Ind}^{G}_{T}\mathcal L}(su)
 =
  \sum_{\mathcal O\in\Gamma^F}
  \chi_{\mathcal L}(x_{\mathcal O}^{-1}sx_{\mathcal O})
  \mathcal{Q}_{T_{\mathcal O},H}^{\mathcal L_{\mathcal O}+}(u).\]
\end{proposition}
\begin{proof}
Let \(u\in H^F\) be unipotent. By
\Cref{sep}, \(u\in\mathcal U\). A nontrivial
Frobenius orbit in \(\Gamma\) contributes zero because Frobenius
permutes the corresponding distinct summands. Thus only
\(\mathcal O\in\Gamma^F\) contribute. For such an \(\mathcal O\),
\Cref{tw-tr} identifies the characteristic function of the corresponding
twisted summand with
\[
 \chi_{\mathcal L}(s_{\mathcal O})
 \mathcal{Q}_{T_{\mathcal O},H}^{\mathcal L_{\mathcal O}+}(u).
\]
Taking characteristic functions in \eqref{glob} proves the
formula. Here the
normalizing shifts on the two sides give the same sign since
$\dim G-\dim H$ is even.
\end{proof}

For \(x\in G^F\) satisfying \(x^{-1}sx\in T\), set
\[
 T_x=xTx^{-1}\subset H,
 \qquad
 c_x:T_x\longrightarrow T,\quad t\longmapsto x^{-1}tx,
 \qquad
 \mathcal L_x=c_x^*\mathcal L.
\]
\begin{theorem}
\label{JD}
For every unipotent element \(u\in H^F\),
\begin{equation}
\chi_{\mathfrak{Ind}^{G}_T\mathcal{L}}(su)
 =\frac{1}{|H^F|}
 \sum_{\substack{x\in G^F\\x^{-1}sx\in T}}
 \mathcal{Q}_{T_x,H}^{\mathcal L_x+}(u)
 \chi_\mathcal{L}(x^{-1}sx).
\label{JD-eq}
\end{equation}
\end{theorem}
\begin{proof}
Recall that
\[
\mathcal M^F=\{x\in G^F:x^{-1}sx\in T\}.
\]
Lang's theorem implies:
\[
 H^F\backslash\mathcal M^F/T^F\xrightarrow{\sim}\Gamma^F.
\]
Indeed, the stabilizer in \(H\times T\) of a representative
\(x_{\mathcal O}\in\mathcal M^F\) is
\[
 \{(x_{\mathcal O}tx_{\mathcal O}^{-1},t):t\in T\}\cong T,
\]
which is connected. So the summand in \eqref{JD-eq} is constant on every
\(H^F\times T^F\)-orbit, right multiplication by \(T^F\) changes
neither \(T_x\), \(\mathcal L_x\), nor \(x^{-1}sx\), while left
multiplication by \(H^F\) conjugates the induction datum in \(H\),
and its characteristic function is a class function on \(H^F\). Moreover,
each such double coset has cardinality
\[
 \frac{|H^F|\,|T^F|}{|T^F|}=|H^F|,
\]
by Lang's Lemma. Therefore, the
right hand side of \eqref{JD-eq} is exactly the sum over
\(\mathcal O\in\Gamma^F\). The result follows from
\Cref{Jsum}.
\end{proof}

\begin{theorem}[Deep level Lusztig's comparison theorem]
\label{LC}
Suppose that $q$ is sufficiently large. Then
\[
 \chi_{\mathfrak{Ind}^{G}_{T}\mathcal L}
 =(-1)^{\dim G}\mathcal R^{G}_{T}(\phi).
\]
\end{theorem}
\begin{proof}
Write
\[
 Q_{T,G}^{\phi+}(v)
 :=\mathcal R_T^G(\phi)(v)
 \qquad (v\in G^F\text{ unipotent}).
\]
By \cite[Theorem 3.1]{LN26}, this function depends only on
$\phi|_{T^{0+}}$. Since $q$ is sufficiently large, we may choose a
regular character $\phi^\dagger$ with the same restriction to $T^{0+}$ as in \Cref{changehowe} and $\mathcal{L}^{\dagger}$ be the corresponding local system on $T$.
Then \Cref{Qpos},
\cite[Theorem 8.6]{INY25}, and 
\cite[Theorem 3.1]{LN26} imply
\begin{equation}\label{Qcomp}
 \begin{aligned}
 \mathcal{Q}_{T,G}^{\mathcal L+}(v)
 &=\mathcal{Q}_{T,G}^{\mathcal L^{\dagger}+}(v)\\
 &=(-1)^{\dim G}\mathcal R_T^G(\phi^\dagger)(v)\\
 &=(-1)^{\dim G} Q_{T,G}^{\phi+}(v).
 \end{aligned}
\end{equation}

Let $g=su\in G^F$ be its Jordan
decomposition and write
$H=Z_G^0(s)$. For $x\in G^F$ with $x^{-1}sx\in T$, let $\phi_x$ be the
character of $T_x^F$ transported from $\phi$. Hence
\eqref{Qcomp} applies to $(H,T_x,\phi_x)$.
Using \Cref{JD}, we obtain
\begin{align*}
 \chi_{\mathfrak{Ind}^{G}_{T}\mathcal L}(su)
 &=\frac{1}{|H^F|}
   \sum_{\substack{x\in G^F\\x^{-1}sx\in T}}
   \mathcal Q_{T_x,H}^{\mathcal{L}_x+}(u)\,
   \phi(x^{-1}sx)\\
&=\frac{(-1)^{\dim H}}{|H^F|}
   \sum_{\substack{x\in G^F\\x^{-1}sx\in T}}
    Q_{T_x,H}^{\phi_x+}(u)\,
   \phi(x^{-1}sx)\\
 &=(-1)^{\dim G}\mathcal R_T^G(\phi)(su).
\end{align*}
as required.
\end{proof}

\section{Orthogonality and simple summands}
\label{endalg}

\subsection{Orthogonality}
In this section we write $d=\dim G$. And 
all local systems are equipped with
Weil structures. We write
\[
\langle f,f'\rangle_{G^F}=|G^F|^{-1}\sum_{g\in G^F}f(g)f'(g).
\]
Let $(T,\mathcal L)$, $(T',\mathcal L')$ be $F$-stable induction pairs,
with corresponding characters $\phi,\phi'$, and set
\[
 \mathcal K=\mathfrak{Ind}_T^G\mathcal L,\qquad
 \mathcal K'=\mathfrak{Ind}_{T'}^G\mathcal L',\qquad
 \mathcal N(T,T')=\{n\in G:nTn^{-1}=T'\}.
\]

\begin{theorem}\label{orth-ind}
One has
\begin{equation}\label{ind-orth}
 \langle \chi_{\mathcal K},\chi_{\mathcal K'}\rangle_{G^F}
 =|\mathcal N(\phi,\phi')^F/T^F|,
\end{equation}
where
\[
 \mathcal N(\phi,\phi')^F
 =\{n\in\mathcal N(T,T')^F:
       \phi(t)\phi'(ntn^{-1})=1, t\in T^F\}.
\]
\end{theorem}
\begin{proof}
For $n\in\mathcal N(T,T')$ write $c_n(t)=ntn^{-1}$, and set
\[
 S=\{nT\in\mathcal N(T,T')/T:
              \mathcal L\otimes c_n^*\mathcal L'\cong
              \overline{\mathbb Q}_\ell\}.
\]
By \cite[\S7.4]{Lus_Ch2}, $H_c^i(G,\mathcal K\otimes\mathcal K')=0$
for $i>0$, whereas its degree zero part is the dual of
$\operatorname{Hom}(\mathcal K,\mathbb D\mathcal K')$.
By
\cite[\S 5]{INY25}, we know this Hom space is linearly spanned by
$\mathcal N(\phi,\phi')/T\cong S$. The relevant
covering space is $G/T\times T_{\mathrm{vreg}}$, and restriction
from $T$ to its dense open $T_{\mathrm{vreg}}$
detects whether the rank-one local system is trivial. Thus
\[
 H_c^0(G,\mathcal K\otimes\mathcal K')
       \cong\overline{\mathbb Q}_\ell[S](-d).
\]
Hence
we have \(\operatorname{Tr}\!\left(
F^m,H_c^0(G,\mathcal K\otimes\mathcal K')
\right)
=q^{md}|S^{F^m}|.\)
Let $I_m$ be the left hand side of \Cref{ind-orth} for $F^m$.
we have
\[
 I_m=|S^{F^m}|+O(q^{-m/2}).
\]
For all sufficiently large $m$, \Cref{LC} makes $I_m$ an integer,
so $I_m=|S^{F^m}|$. The difference
\[
 \operatorname{Tr}(F^m,R\Gamma_c(G,\mathcal K\otimes\mathcal K'))
       -|G^{F^m}|\operatorname{Tr}(F^m,\overline{\mathbb Q}_\ell[S])
\]
is a finite sum of powers of nonzero eigenvalues. It is vanishing for every $m$ not only large $m$, by Vandermonde
argument. In particular $I_1=|S^F|$. Lang's theorem for $T$ gives
an $F$-fixed representative of every coset in $S^F$, then identifies $S^F$ with
$\mathcal N(\phi,\phi')^F/T^F$.
\end{proof}
\begin{remark}
    This theorem can be also proved as follows: first, by Mackey formula of deep level Deligne-Lusztig character, we know \Cref{ind-orth} holds for sufficiently large $m$. Then we still use Vandermonde
argument to both side.
\end{remark}
\begin{theorem}\label{orth-green}
Write $\phi^+=\phi|_{(T^{0+})^F}$ and similarly for $\phi'^+$.
Then
\begin{equation}\label{green-orth}
\begin{aligned}
 &\frac1{|G^F|}\sum_{u\in G^F_{\mathrm{unip}}}
 \mathcal Q_{T,G}^{\mathcal L+}(u)\mathcal Q_{T',G}^{\mathcal L'+}(u)\\
 &=\frac1{|T^F|\,|(T')^F|}
 \sum_{n\in\mathcal N(T,T')^F}\sum_{t\in(T^{0+})^F}
                    \phi^+(t)\phi'^+(ntn^{-1}).
\end{aligned}
\end{equation}
\end{theorem}
\begin{proof}
From \Cref{changehowe}, \Cref{Qpos}, \Cref{JD}, we obtain 
\[
 \sum_{\lambda\in\widehat{T^F_0}}
 \chi_{\mathfrak{Ind}_T^G\mathcal L_{\phi\otimes{\lambda}}}(g)
 =\begin{cases}
 |T^F_0|\mathcal Q_{T,G}^{\mathcal L+}(g),&g\text{ unipotent},\\
 0,&\text{otherwise}.
 \end{cases}
\]
Indeed, the sum of the characters $\lambda$ vanishes at every
nonidentity semisimple element of $T_0^F$. Pair this identity with
$\chi_{\mathcal K'}$ and use \Cref{orth-ind}.
$nT^F\in\mathcal N(T,T')^F/T^F$ contributes exactly once only if
$\phi^+c_n^*\phi'^+=1$.
Consequently the required scalar product is $|T^F_0|^{-1}$ times the
number of these cosets. In (\ref{green-orth}) the inner sum is
$|(T^{0+})^F|$ for precisely the same cosets and zero otherwise.
Since $|T^F|=|T^F_0|\,|(T^{0+})^F|=|(T')^F|$ whenever $\mathcal{N}(T,T')^F$ is nonempty, the two expressions agree. If
$N(T,T')^F$ is empty, both are zero.
\end{proof}

\subsection{Simple induced sheaves}
Let $\mathcal K=\mathfrak{Ind}_T^G\mathcal L$ and write
$W_{\mathcal L}=\{w\in N_G(T)/T:w^*\mathcal L\cong\mathcal L\}.
$.
By
\cite[\S 5]{INY25}, we have:
\begin{equation}\label{isotypic}
\begin{gathered}
 \mathscr A:=\operatorname{End}(\mathcal K)
       \cong\overline{\mathbb Q}_\ell[W_{\mathcal L}],\\
 \mathcal K\cong\bigoplus_A A\otimes V_A,\quad
 V_A=\operatorname{Hom}(A,\mathcal K).
\end{gathered}
\end{equation}
Here $A$ runs through the simple summands of $\mathcal K$, and $V_A$ through the
irreducible $\mathscr A$-modules. For $w\in W_{\mathcal{L}}$, let $\theta_w$ be the standard
basis from the deck transformation
$
(g,xT)\mapsto(g,x\dot w^{-1}T),
$
thus $\theta_w\theta_v=\theta_{wv}$.
Similar to \cite[\S10.4]{Lus_Ch2}, for $v\in V_A$ let $F^*(v)$ be the corresponding homomorphism $F^*A\to F^*\mathcal{K}$. Set
\[
 \iota(a)=\Phi F^*(a)\Phi^{-1},\qquad
 \sigma_A(v)=\Phi F^*(v)\Phi_A^{-1},
\]
where $\iota$ is the automorphism of the algebra $\mathscr{A}$ defined as above, $\Phi:F^*\mathcal K\simeq\mathcal K$ is the original Weil structure and
$\Phi_A:F^*A\simeq A$ is chosen for each $F$ equivariant summand.
Then $\sigma_A(av)=\iota(a)\sigma_A(v)$ and $\iota(\theta_v)=\theta_{F^{-1}(v)}$. Set
\[
 h_A(w)=\operatorname{Tr}(\theta_w\sigma_A,V_A),\qquad
 h_A^\vee(w)=\operatorname{Tr}((\theta_w\sigma_A)^{-1},V_A).
\]
Then by two orthogonality relations in \cite[\S10.3]{Lus_Ch2}, we have
\begin{equation}\label{twisted-schur}
\begin{aligned}
 \frac1{|W_{\mathcal L}|}\sum_w h_A(w)h_B^\vee(w)&=\delta_{A,B},\\
 \sum_A h_A(w)h_A^\vee(w')
 &=\operatorname{Tr}\bigl(a\mapsto\theta_w\iota(a)\theta_{w'}^{-1},
                         \mathscr A\bigr).
\end{aligned}
\end{equation}
Only $F$-equivariant summands occur in these formulas.

\begin{proposition}
Fix a Howe factorization of $\phi$, and set
\[
 H_0=(G^0)_0,\qquad
 \mathcal K_0=\mathfrak{Ind}_{T_0}^{H_0}\mathcal L_{\phi_{-1}}.
\]
Let $\Psi^+=\Psi_n^\dag \cdots \Psi_1^\dag$ denote the composition of the operations defined in \cite[\S4.2]{INY25}, then
$\mathcal K\cong\Psi_+(\mathcal K_0)$, and this functor gives
isomorphisms
\[
 \operatorname{End}(\mathcal K_0)\xrightarrow{\sim}\mathscr A,\qquad
 j_A:\operatorname{Hom}(A_0,\mathcal K_0)
       \xrightarrow{\sim}\operatorname{Hom}(A,\mathcal K),
 \quad A=\Psi_+(A_0),
\]
where $A_0$ runs through the simple summands of $\mathcal K_0$.
\end{proposition}
\begin{proof}
See \cite[\S4]{INY25}. Note the functors $\Psi_i^\dag$ for $1 \le i \le n$ are fully faithful.
\end{proof}

\begin{remark}\label{depth-zero-traces}
Choose the Weil structures on $A$ and $\mathcal K$ compatibly with
those on $A_0$ and $\mathcal K_0$. satisfying
$j_A\theta_w^0=\theta_wj_A$ and $j_A\sigma_{A_0}=\sigma_Aj_A$. Then
\[
 h_A^\vee(w)
 =\operatorname{Tr}\bigl((\theta_w^0\sigma_{A_0})^{-1},
                  \operatorname{Hom}(A_0,\mathcal K_0)\bigr).
\]
\end{remark}

\begin{proposition}\label{fourier-prop}
For every $F$-equivariant simple summand $A$ of $\mathcal{K}$,
\begin{equation}\label{fourier}
 \chi_{A,\Phi_A}
 =\frac1{|W_{\mathcal L}|}\sum_{w\in W_{\mathcal L}}
          h_A^\vee(w)\chi_{\mathcal K,\theta_w\circ\Phi}.
\end{equation}
When $q$ is sufficiently large,
\begin{equation}\label{DL-fourier}
 \chi_{A,\Phi_A}
 =\frac{(-1)^d}{|W_{\mathcal L}|}\sum_{w\in W_{\mathcal L}}
h_A^\vee(w)\mathcal R_{T_w}^G(\phi_w),
\end{equation}
\begin{equation*}
 \mathcal R_{T_w}^G(\phi_w)
 =(-1)^d\sum_{A\in\operatorname{Irr}(\mathcal K)^F}
   h_A(w)\chi_{A,\Phi_A}
\end{equation*}
where $z_w\in G$ with $z_w^{-1}F(z_w)=\dot w^{-1}$, and $(T_w,\phi_w)$ is obtained from $(T,\phi)$ conjugated by $z_w$.

\end{proposition}
\begin{proof}
Taking traces on both sides of (\ref{isotypic}) gives
$\chi_{\mathcal K,\theta_w\Phi}=\sum_Bh_B(w)\chi_{B,\Phi_B}$.
Multiply by $h_A^\vee(w)$, sum over $w$, and use
(\ref{twisted-schur}) we obtain (\ref{fourier}). As in \cite[\S10.6]{Lus_Ch2}, conjugation
by $z_w$ identifies $(\mathcal K,\theta_w\Phi)$ with induction
from the transported induction pair. By \Cref{LC} the second
and third formula follow.
\end{proof}

\begin{corollary}
Assume $q$ is sufficiently large.
Write $d_0=\dim G^0_0$ and $K=\mathcal K_{\phi,r}$, and let $\kappa_\phi$
denote the character of $K^F$ as in
\cite[\S7.1]{Nie_24}. Then we have
\begin{equation*}
 \chi_{A,\Phi_A}
 =(-1)^{d-d_0}\operatorname{Ind}_{K^F}^{G^F}
       \bigl(\kappa_\phi\otimes\operatorname{Inf}_{(G^0_0)^F}^{K^F}
                                \chi_{A_0,\Phi_{A_0}}\bigr),
\end{equation*}
\end{corollary}
\begin{proof}
Write
$a_w=|W_{\mathcal L}|^{-1}h_A^\vee(w)$.
We have the two Fourier expansions
\[
 \chi_A=(-1)^d\sum_w a_w\mathcal R_{T_w}^G(\phi_w),\qquad
 \chi_{A_0}=(-1)^{d_0}\sum_w a_wR_{T_{0,w}}^{H_0}(\phi_{-1,w}).
\]
By \cite[Theorem~7.5]{Nie_24}, we have
\[
 \mathcal R_{T_w}^G(\phi_w)
 =\operatorname{Ind}_{K^F}^{G^F}
       (\kappa_\phi\otimes\operatorname{Inf}_{H_0^F}^{K^F}R_{T_{0,w}}^{H_0}(\phi_{-1,w})).
\]
The proof is done.
\end{proof}

\subsection{Orthogonality of simple summands}
Write $\mathbb D^0A=(\mathbb DA)(-d)$, and
write $\Phi_A^\vee$ for the corresponding dual structure. Recall
from \cite[\S7.4]{Lus_Ch2} that for simple induced sheaves
\begin{equation}\label{geom-orth}
\begin{gathered}
 H_c^i(G,A\otimes B)=0\quad(i>0),\\
 H_c^0(G,A\otimes B)=0\ \Longleftrightarrow\ B\not\cong\mathbb DA.
\end{gathered}
\end{equation}
In the matching case $B\cong\mathbb{D}A$ the degree-zero space is one-dimensional by Schur's lemma.

Write $f_w=\chi_{\mathcal K,\theta_w\Phi}$ and let $f_{w'}^\vee$
be the function of $\mathbb D^0\mathcal K$ with the structure
dual to $\theta_{w'}\Phi$. Applying (\ref{ind-orth}) to the rational
twists gives, as in \cite[\S10.6]{Lus_Ch2},
\begin{equation}\label{twisted-pairing}
\begin{aligned}
 \langle f_w,f_{w'}^\vee\rangle_{G^F}
 &=\operatorname{Tr}\bigl(a\mapsto\theta_w\iota(a)\theta_{w'}^{-1},
                         \mathscr A\bigr)\\
 &=\sum_C h_C(w)h_C^\vee(w').
\end{aligned}
\end{equation}
Indeed, the matching rational transporter cosets are the basis
elements fixed by $\theta_v\mapsto\theta_w\iota(\theta_v)\theta_{w'}^{-1}$;
unit normalization gives each fixed element coefficient one.
The second equality is (\ref{twisted-schur}).

\begin{theorem}\label{simple-orth}
Let $A_1,A_2$ be $F$-equivariant simple induced sheaves. Then
\[
 \langle \chi_{A_1,\Phi_1},\chi_{A_2,\Phi_2}\rangle_{G^F}
 =\begin{cases}
 0,&A_2\not\cong\mathbb D^0A_1,\\
 1,&(A_2,\Phi_2)=(\mathbb D^0A_1,\Phi_1^\vee).
 \end{cases}
\]
Moreover, If $\Phi_2=c\Phi_1^\vee$ under a geometric identification, the
second value is $c$.
\end{theorem}
\begin{proof}
If the inductions pair for $A_1$ and $\mathbb D^0A_2$ are not
geometrically conjugate, by \Cref{fourier-prop} the left hand side is $0$. We may assume they have a common induced
sheaf complex $\mathcal K$, and write $A_1=A$, $A_2=\mathbb D^0B$
with the dual structure. The two expansions are
\[
\begin{aligned}
 \chi_{A,\Phi_A}&=|W_{\mathcal L}|^{-1}\sum_w h_A^\vee(w)f_w,\\
 \chi_{\mathbb D^0B,\Phi_B^\vee}
       &=|W_{\mathcal L}|^{-1}\sum_{w'}h_B(w')f_{w'}^\vee.
\end{aligned}
\]
Insert (\ref{twisted-pairing}) and apply (\ref{twisted-schur})
first in $w$ and then in $w'$. The result is $\delta_{A,B}$. Scaling the second structure
gives the final assertion.
\end{proof}

\subsection{A perspective toward almost characters on parahoric subgroups}

Set
$$
\mathcal C(G^F)
=
\{f:G^F\to\overline{\mathbb Q}_\ell
\mid
f(xgx^{-1})=f(g),\ \forall x,g\in G^F\}.
$$

Let
$
\mathcal C^{\mathrm{DL}}(G^F)\subset \mathcal C(G^F)
$
be the subspace spanned by the irreducible constituents of deep level Deligne--Lusztig characters, and let
$
\mathcal C^{\mathrm{CS}}(G^F)\subset \mathcal C(G^F)
$
be the subspace spanned by the characteristic functions of $F$-equivariant simple induced sheaves.

Note that, in contrast to the classical setting, $\mathcal C^{\mathrm{DL}}(G^F)$ need not coincide with the full space $\mathcal C(G^F)$. Based on Lusztig's explicite computations in \cite{Lusztig_04} Stasinski showed in \cite{Stansinski} that certain "nilpotent" representations of $\mathrm{SL}_2(\mathbb{F}_q[t]/t^2)$ were unrealisable from irreducible components of deep level Deligne-Lusztig character.

Thus, although \Cref{LC} identifies the characteristic function of  character sheaves with deep level Deligne--Lusztig characters, it does not by itself determine the relation between $\mathcal C^{\mathrm{CS}}(G^F)$ and $\mathcal C^{\mathrm{DL}}(G^F)$.

Fortunately, for a fixed $F$-stable induction pair and sufficiently
large $q$, \eqref{DL-fourier} gives
\[
\operatorname{span}
   \{\chi_{A,\Phi_A}:A\in\operatorname{Irr}(\mathcal K)^F\}
 =\operatorname{span}\{\mathcal R_{T_w}^G(\phi_w):w\in W_{\mathcal L}\},.
\]
Thus the simple induced sheaf functions give an orthogonal basis for $\mathcal{C}^{\mathrm{UN}}(G^F)$.

This suggests considering more general induction constructions (like in \cite{Lus_Ch2}) from some Levi data, may enlarge $\mathcal C^{\mathrm{CS}}(G^F)$ to the full space
$
\mathcal C^{\mathrm{CS},\,\mathrm{gen}}(G^F)
=
\mathcal C^{\mathrm{DL}}(G^F)
$
, and the characteristic function $\chi_A$ of simple general induced sheaves form another orthogonal basis for $\mathcal C^{\mathrm{UN}}(G^F)$. Which implies
$
\mathcal C^{\mathrm{UN}}(G^F)=\mathcal C^{\mathrm{CS}}(G^F)
\subset
\mathcal C^{\mathrm{DL}}(G^F).
$

Moreover, in analogy with the classical theory of almost characters, It would be interesting to realize $\chi_A$ by a combinatorial construction.

\section{Application: $G^F$-invariants and symmetric multiplicities}
\label{quad}
We now apply the comparison theorem to the distinction problem for the
subfield pair $(G^{F^2},G^F)$. For a representation $\rho$ of $G^{F^2}$,
the quantity of interest is
\[
 \dim\operatorname{Hom}_{G^F}(\rho,1)
 =\dim\rho^{G^F}
 =\frac{1}{|G^F|}\sum_{g\in G^F}\operatorname{Tr}(g,\rho).
\]
In this section, the character sheaves we consider are $F^2$-equivariant. 
\subsection{}
Let $(\mathcal A,\Phi_{\mathcal A})$ be an $F^2$-equivariant sheaf complex on
$G$, i.e, $\Phi_{\mathcal A}:(F^2)^*\mathcal A\cong \mathcal A$. We define an isomorphism $\widetilde\Phi_{\mathcal A}:F^*(F^*\mathcal A\otimes\mathcal A)\to F^*\mathcal A\otimes\mathcal A$ as
the $F$-Weil structure
\begin{equation}\label{swap}
 F^*(F^*\mathcal A\otimes\mathcal A)
 \xrightarrow{\Phi_{\mathcal A}\otimes1}
 \mathcal A\otimes F^*\mathcal A
 \xrightarrow{c}F^*\mathcal A\otimes\mathcal A,
\end{equation}
where $c$ is the derived interchange. 

\begin{lemma}\label{trace-swap-lemma}
We have
\begin{equation*}
 \sum_{g\in G^F}\chi_{\mathcal A,\Phi_{\mathcal A}}(g)
       =\operatorname{Tr}(\widetilde\Phi_{\mathcal A},R\Gamma_c(G,F^*\mathcal A\otimes\mathcal A)).
\end{equation*}
\end{lemma}
\begin{proof}
Consider a stalk $\mathcal{A}_g$, here $g\in G^F \subset G^{F^2}$. We have 
\begin{align*}
(\widetilde\Phi_{\mathcal A})_g:\mathcal{A}_{F^2(g)}\otimes \mathcal{A}_{F(g)}\to \mathcal{A}_{F(g)}\otimes \mathcal{A}_{g}\\
v\otimes w \mapsto (-1)^{|v||w|}w\otimes (\Phi_{\mathcal A})_g(v)
\end{align*}
hence the (derived) trace of
$(\widetilde\Phi_{\mathcal A})_g$ on $\mathcal{A}_g\otimes \mathcal{A}_g$ is the trace of $(\Phi_{\mathcal A})_g$ on $\mathcal{A}_g$ as required. 
\end{proof}
Throughout this section,
let $A$ be an $F^2$-equivariant simple induced sheaf with $F^2$-equivariant Weil structure $\Phi_{A}$ and write $A|_{G_{\mathrm{vreg}}}=E[d]$ which is a local system.

Suppose $F^*A\cong\mathbb D^0A$. Choose
$\alpha:F^*A\simeq\mathbb D^0A$ and define
\begin{equation*}
 \Phi_A^{\mathrm{can}}
       =\mathbb D^0(\alpha^{-1})\circ F^*\alpha:
                  (F^2)^*A\longrightarrow A.
\end{equation*}
The restriction of
\(\alpha\) is an isomorphism
\(
\alpha_{|E}:F^*E\xrightarrow{\sim}E^\vee.
\)
Let \(\delta_E:E\xrightarrow{\sim}E^{\vee\vee}\) be the standard
isomorphism. We have
\(
\Phi_{|E}^{\mathrm{can}}
=
\delta_E^{-1}\circ(\alpha^{-1}_{|E})^\vee\circ F^*\alpha_{|E}
\)
where we use the natural identification
\(F^*(E^\vee)\simeq(F^*E)^\vee\).
The isomorphism \(\Phi_A^{\mathrm{can}}\) is defined as the unique
middle extension of \(\Phi_E^{\mathrm{can}}[d]\). 

Since \(A\) is simple, any \(F^2\)-equivariant structure on \(A\)
differs from \(\Phi_A^{\mathrm{can}}\) by a scalar. After rescaling,
we may therefore assume that
\(\Phi_A=\Phi_A^{\mathrm{can}}\).

\begin{proposition}\label{F2-simple}
Suppose $q$ is sufficiently large, We have
\[
\begin{aligned}
 |G^F|^{-1}\sum_{g\in G^F}\chi_{A,\Phi_A}(g)&=0&&\text{if }F^*A\not\cong\mathbb DA,\\
 |G^F|^{-1}\sum_{g\in G^F}\chi_{A,\Phi_A}(g)&=(-1)^d
       &&\text{if }F^*A\cong\mathbb DA.
\end{aligned}
\]
\end{proposition}
\begin{proof}
By (\ref{geom-orth}), $H_c^0(G,F^*A\otimes A)$ is zero in the
nondual case and one-dimensional in the dual case. In the latter,
$\widetilde\Phi_{\mathcal A}$ acts on this one-dimensional space by $(-1)^dq^d$. Since $F^*A\otimes A$ has weights at most
$2d$, every eigenvalue $\lambda$ of $\widetilde\Phi_A$ on
$H_c^i(G,F^*A\otimes A)$ satisfies
\(|\lambda|^2\le(q^2)^{(2d+i)/2},
 |\lambda|\le q^{d+i/2}\)
for every complex absolute value. Set $\delta=0$ if $F^*A\not\simeq\mathbb DA$ and $\delta=(-1)^d$ if $F^*A\simeq\mathbb DA$. By \Cref{trace-swap-lemma}
\[
 \left|\sum_{g\in G^F}\chi_{A,\Phi_A}(g)-\delta q^d\right|
 \le q^{d-1/2}\sum_{i<0}\dim H_c^i(G,F^*A\otimes A).
\]
And $q$ is sufficiently large so that
\[
 |W_{\mathcal L}|\left(
 q^{d-1/2}\sum_{i<0}\dim H_c^i(G,F^*A\otimes A)
 +\bigl|q^d-|G^F|\bigr|\right)<|G^F|.\]
we have
\[
 \left|\frac1{|G^F|}\sum_{g\in G^F}\chi_{A,\Phi_A}(g)-\delta\right|
 <|W_{\mathcal L}|^{-1}.
\]
Apply \eqref{DL-fourier} with Frobenius $F^2$.  The
local system $E$ has finite arithmetic monodromy hence
the action of $W_{\mathcal L}$, so
$\theta_w\sigma_A$ has finite order.  Therefore
$h_A^\vee(w)$ is a sum of roots of unity.  Moreover,
\(
 \frac{1}{|G^F|}\sum_{g\in G^F}
 \mathcal R_{T_w}^{G,F^2}(\phi_w)(g)
 =\left\langle
 \operatorname{Res}_{G^F}^{G^{F^2}}
 \mathcal R_{T_w}^{G,F^2}(\phi_w),1
 \right\rangle_{G^F}\in\mathbb Z
\).
It follows that
\[
 |W_{\mathcal L}|\left(
 \frac1{|G^F|}\sum_{g\in G^F}\chi_{A,\Phi_A}(g)-\delta\right)
\]
is an algebraic integer has
absolute value less than one, it must be zero.
\end{proof}

\subsection{The subfield symmetry multiplicity }
Let $(T,\mathcal L)$ be $F^2$-equivariant, with normalized trace character
$\phi:T^{F^2}\to\overline{\mathbb Q}_\ell^\times$, and let
$(\mathcal K,\Phi)$ be the induced complex. Set
\[
 m_{T,\phi}=\left\langle
 \operatorname{Res}_{G^F}^{G^{F^2}}\mathcal R_T^{G,F^2}(\phi),1
 \right\rangle_{G^F}.
\]
For each $F^2$-equivariant simple summand choose $\Phi_A$, adapt
(\ref{swap}) when $F^*A\cong\mathbb D^0A$ and canonical Weil structure, and write
$\sigma_A(v)=\Phi(F^2)^*(v)\Phi_A^{-1}$ on $V_A$.

\begin{theorem}\label{subfield-trace}
When $q$ is sufficiently large,
\begin{equation*}
\begin{aligned}
 m_{T,\phi}
 &=\frac{(-1)^d}{|G^F|}
   \operatorname{Tr}(\widetilde\Phi,
                  R\Gamma_c(G,F^*\mathcal K\otimes\mathcal K))\\
 &=\sum_{\substack{A\in\operatorname{Irr}(\mathcal K)^{F^2}\\
                      F^*A\cong\mathbb D^0A}}
                      \operatorname{Tr}(\sigma_A,V_A).
\end{aligned}
\end{equation*}
In particular,
\begin{equation*}
 |m_{T,\phi}|\leq\sum_A\dim V_A
                 \leq |W_{\mathcal L}|.
\end{equation*}
\end{theorem}
\begin{proof}
By comparison theorem $\chi_{\mathcal K,\Phi}=(-1)^d\mathcal R_T^{G,F^2}(\phi)$,
so (\ref{trace-swap-lemma}) implies the first equality. Taking traces in
(\ref{isotypic}) and using \Cref{F2-simple} proves the second and 
not $F^2$-equivariant summands have trace zero (permutation action).
The induced Weil structure have finite monodromy on the very regular locus, so the contributing
$\sigma_A$ have finite order. Finally,
\[
 |\operatorname{Tr}(\sigma_A,V_A)|\leq\dim V_A,\qquad
 \sum_A(\dim V_A)^2=\dim\mathscr A=|W_{\mathcal L}|.
\]
These imply the results.
\end{proof}

\begin{corollary}\label{regular-F2}
If $\mathcal L$ is regular, then $|m_{T,\phi}|\leq1$.
$m_{T,\phi}=0$ if
$F^*\mathcal K\not\cong\mathbb D^0\mathcal K$ and 
$m_{T,\phi}=\pm1$ if $F^*\mathcal K\cong\mathbb D^0\mathcal K$.
Write $\rho=\pm\mathcal R_T^{G,F^2}(\phi)$ as the corresponding irreducible
representation, then $\dim\rho^{G^F}$ is respectively $0$ or $1$.
\end{corollary}
\begin{proof}
Now $\mathcal K$ is simple and the single multiplicity space is
one dimensional. Its operator in the matching case is a nonzero
root of unity.
\end{proof}

\begin{corollary}\label{central-obstruction}
If $\phi|_{Z(G)^F\cap T}$ is nontrivial, then $m_{T,\phi}=0$.
\end{corollary}
\begin{proof}
The left and right actions of $z\in Z(G)^F\cap T$ on the defining
variety agree. Thus $z$ acts by $\phi(z)$ on each cohomology group
in the $\phi$-isotypic part. If $\phi(z)\ne1$, each space of
$G^F$-invariants is zero, and so is its alternating multiplicity.
\end{proof}

\bibliography{bib_ADLV}{}
\bibliographystyle{amsalpha}
\end{document}